\documentclass[11pt]{article}
\usepackage{hyperref, amsfonts, amsmath, amssymb, amsthm, fullpage, enumitem, tikz, mathtools, microtype}
\usepackage[noabbrev, capitalize]{cleveref}

\newtheorem{theorem}{Theorem}
\newtheorem{corollary}[theorem]{Corollary}
\newtheorem{lemma}[theorem]{Lemma}
\newtheorem{proposition}[theorem]{Proposition}
\newtheorem{problem}[theorem]{Problem}
\newtheorem{conjecture}[theorem]{Conjecture}

\theoremstyle{definition}
\newtheorem{definition}[theorem]{Definition}

\theoremstyle{remark}
\newtheorem*{remark}{Remark}

\newtheorem*{claim*}{Claim}

\crefname{claim}{Claim}{Claims}

\DeclarePairedDelimiter{\abs}{\lvert}{\rvert}
\newcommand{\la}{\lambda}
\newcommand{\las}{\lambda^*}
\newcommand{\eps}{\varepsilon}
\newcommand{\N}{\mathbb{N}}
\newcommand{\F}{\mathcal{F}}
\newcommand{\G}{\mathcal{G}}
\newcommand{\cH}{\mathcal{H}}

\newcommand{\R}{\mathbb{R}}
\newcommand{\Gp}{\mathcal{G}_p}
\newcommand{\nabd}{N_{\alpha,\beta}(d)}
\newcommand{\mult}{\mathrm{mult}}
\newcommand{\sset}[1]{\left\{#1\right\}}
\newcommand{\dset}[2]{\sset{#1 \colon #2}}
\newcommand{\rowingrefs}{(\ref{eqn:rowing-b},\ref{eqn:rowing-b-1},\ref{eqn:rowing-b-2})}
\newcommand{\rowingdrefs}{(\ref{eqn:rowing-d-1}, \ref{eqn:rowing-d-2})}

\usetikzlibrary{calc}
\definecolor{litegray}{RGB}{192,192,192}
\tikzstyle{vertex}=[circle,draw,fill=white,inner sep=0.5pt,minimum width=4pt]
\tikzstyle{root-vertex}=[circle,draw,fill=darkgray,inner sep=0.5pt,minimum width=4pt,text=white]

\title{On the smallest eigenvalues of $3$-colorable graphs}

\usepackage{ifthen}  

\newboolean{anonymize}
\setboolean{anonymize}{false}  

\ifthenelse{\boolean{anonymize}}{
  \author{Anonymized for peer review}
}{
\author{
    Zilin Jiang\thanks{School of Mathematical and Statistical Sciences, and School of Computing and Augmented Intelligence, Arizona State University, Tempe, AZ 85281. Email: {\tt zilinj@asu.edu}. Supported in part by the Simons Foundation through its Travel Support for Mathematicians program and by U.S.\ taxpayers through NSF grant 2451581.}
    \and Zhiyu Wang\thanks{Department of Mathematics, Louisiana State University, Baton Rouge, LA 70803. Email: {\tt zhiyuw@lsu.edu}. Supported in part by the Louisiana Board of Regents through grant LEQSF(2024-27)-RD-A-16.}
}
}
\date{}
\begin{document}

\maketitle

\begin{abstract}
    We prove that the set of the smallest eigenvalues attained by $3$-colorable graphs is dense in $(-\infty, -\lambda^*)$, where $\lambda^* = \rho^{1/2} + \rho^{-1/2} \approx 2.01980$ and $\rho$ is the positive real root of $x^3 = x + 1$. As a consequence, in the context of spherical two-distance sets, our result precludes any further refinement of the forbidden-subgraph method through the chromatic number of signed graphs.
\end{abstract}

\section{Introduction}

Motivated by the problem of estimating the maximum size of spherical two-distance sets with prescribed inner products, Jiang and Polyanskii \cite{JP24} studied certain families of signed graphs within the class $\Gp$ of $p$-colorable signed graphs. Here, a \textit{signed graph} is a graph $G$ whose edges are each labeled by $+$ or $-$, and a signed graph $G$ is \emph{$p$-colorable} if there exists a $p$-coloring of $V(G)$ such that the endpoints of every negative edge receive identical colors, and the endpoints of every positive edge receive different colors.

\begin{remark}
    Our coloring notion is the sign-reversed version of Cartwright--Harary \cite{CH68}, where negative edges must join vertices of \emph{different} colors and positive edges of \emph{identical} colors. We choose the current definition because we need to discuss the (unsigned) graph coloring as well, and the current definition is a generalization of that. By flipping the signs of all edges, one can go back and forth between the two different definitions, as we will do when restating various results from \cite{JP24,JTYZZ23}.
\end{remark}

We now introduce the key families of signed graphs. Given $\la \in \R$, let $\Gp(\la)$ denote the family of $p$-colorable signed graphs whose smallest eigenvalues are at least $-\la$. The \emph{smallest eigenvalue} of a signed graph $G$, denoted by $\la_1(G)$, is the smallest eigenvalue of its signed adjacency matrix.

The Cauchy interlacing theorem implies that $\Gp(\la)$ is closed under taking induced subgraphs. Jiang and Polyanskii asked in \cite{JP24} whether it is possible to define each of these families by a \emph{finite} set of forbidden subgraphs.

\begin{definition}
    Given a class $\G$ of signed graphs, and a family $\cH \subseteq \G$ that is closed under taking induced subgraphs, a set $\F \subseteq \G$ of signed graphs is a \textit{forbidden subgraph characterization} of $\cH$ within $\G$ if $\cH$ consists exactly of signed graphs in $\G$ that do not contain any member of $\F$ as an induced subgraph.    
\end{definition}

\begin{problem}[Problem 5.3 of Jiang and Polyanskii \cite{JP24}] \label{prob:main}
    For every $p \in \N^+$, determine the set of $\la \in \R$ for which $\Gp(\la)$ has a finite forbidden subgraph characterization within $\Gp$.
\end{problem}

For $p \in \sset{1,2}$, \cref{prob:main} is essentially resolved by \cite[Theorem~1]{JP20}. In this paper, we completely resolve the rest of \cref{prob:main} as follows.

\begin{theorem} \label{thm:main-forb}
    For every $p \ge 3$, within the class $\Gp$ of $p$-colorable signed graphs, the family $\Gp(\la)$ of signed graphs with smallest eigenvalues at least $-\la$ has a finite forbidden subgraph characterization if and only if $\la < \las$, where
    \[
        \las := \rho^{1/2} + \rho^{-1/2} \approx 2.01980,
    \]
    and $\rho$ is the unique real root of $x^3 = 1 + x$.
\end{theorem}

The key ingredient is the following construction of $3$-colorable graphs.

\begin{theorem} \label{thm:main}
    For every $\la > \las$ and every $\eps > 0$, there exists a $3$-colorable graph $G$ such that $\la_1(G) \in (-\la - \eps, -\la)$.
\end{theorem}

The rest of the paper is organized as follows. \cref{sec:forb} derives \cref{thm:main-forb} from \cref{thm:main}. \cref{sec:rowing} proves \cref{thm:main}, deferring three technical results to \cref{sec:technical}. Finally, \cref{sec:two-distance} revisits the spherical two-distance set motivation.

\section{Forbidden subgraph characterization} \label{sec:forb}

We first briefly explain why \cref{prob:main} for $p \in \sset{1,2}$ is essentially resolved before.

\begin{theorem}[Theorem 1 of Jiang and Polyanskii \cite{JP20}] \label{thm:forb-sprad}
    For every integer $m \ge 2$, let $\beta_m$ be the largest root of $x^{m+1} = 1 + x + \dots + x^{m-1}$, and let $\alpha_m := \beta_m^{1/2} + \beta_m^{-1/2}$. The family of graphs with spectral radius at most $\la$ has a finite forbidden subgraph characterization (within the class of graphs) if and only if $\la < \la'$ and $\la \notin \sset{\alpha_2, \alpha_3, \dots}$, where
    \[
        \la' := \lim_{m\to\infty} \alpha_m = \varphi^{1/2} + \varphi^{-1/2} = \sqrt{2 + \sqrt{5}} \approx 2.05817,
    \]
    and $\varphi$ is the golden ratio $(1+\sqrt5)/2$. \qed
\end{theorem}

To reduce the case where $p = 2$ to $p = 1$, we rely on a useful tool in spectral graph theory for signed graphs --- two signed graphs are \emph{switching equivalent} if one graph can be obtained from the other by reversing all the edges in a cut-set. An important feature of switching equivalence is that the switching equivalent signed graphs all have the same spectrum.

\begin{corollary}
    For $p \in \sset{1,2}$, within the class $\Gp$ of $p$-colorable signed graphs, the family $\Gp(\la)$ of signed graphs with smallest eigenvalues at least $-\la$ has a finite forbidden subgraph characterization if and only if $\la < \la'$ and $\la \not\in \sset{\alpha_2, \alpha_3, \dots}$.
\end{corollary}

\begin{proof}
    Notice that for every signed graph $G$, it is $1$-colorable if and only if $-G$ is unsigned, and $\la_1(G) \ge -\la$ if and only if the largest eigenvalue of $-G$ is at most $\la$. Under this correspondence between $\G_1$ and the class of unsigned graphs, the family $\G_1(\la)$ has a finite forbidden subgraph characterization within $\G_1$ if and only if the family of graphs with spectral radius at most $\la$ has a finite forbidden subgraph characterization. Therefore, the case where $p = 1$ reduces to that in \cref{thm:forb-sprad}.

    By restricting signed graphs in $\G_2$ to $\G_1$, we observe that if $\G_2(\la)$ has a finite forbidden subgraph characterization within $\G_2$, then so does $\G_1(\la)$ within $\G_1$. Conversely, suppose that $\G_1(\la)$ has a finite forbidden subgraph characterization $\F_1$ within $\G_1$. Let $\F_2$ be the family of signed graphs that are switching equivalent to signed graphs in $\F_1$. Using the fact that every $2$-colorable signed graph is switching equivalent to a $1$-colorable signed graph, one can check that $\F_2$ is a finite forbidden subgraph characterization of $\G_2(\la)$ within $\G_2$.
\end{proof}

The rest of the section focuses on solving \cref{prob:main} for $p \ge 3$.

\begin{theorem}[Theorem 1.5 of Jiang and Polyanskii \cite{JP24}] \label{thm:signed-forb}
    The family $\G(\la)$ of signed graphs with smallest eigenvalue at least $-\la$ has a finite subgraph characterization (within the class of signed graphs) if and only if $\la < \las$. \qed
\end{theorem}

\begin{proof}[Proof of \cref{thm:main-forb} for $\la < \las$]
    In view of \cref{thm:signed-forb}, we know that the family $\G(\la)$ has a finite forbidden subgraph characterization. By restricting signed graphs to those in $\Gp$, we obtain a finite forbidden subgraph characterization of $\Gp(\la)$ within $\Gp$ from that of $\G(\la)$.
\end{proof}

\begin{proof}[Proof of \cref{thm:main-forb} for $\la \ge \las$ assuming \cref{thm:main}]
    Assume for the sake of contradiction that $\F$ is a finite forbidden subgraph characterization of $\Gp(\la)$ within $\Gp$. Because every signed graph in $\Gp \setminus \Gp(\la)$ contains a member of $\F$ as an induced subgraph, no $p$-colorable signed graph has its smallest eigenvalue in the open interval $(\max\dset{\la_1(F)}{F \in \F}, -\la)$, which contradicts with \cref{thm:main}.
\end{proof}

\section{Binary rowing graphs} \label{sec:rowing}

A large chunk of \cref{thm:main} is essentially established by Shearer \cite{S89}, who proved that the set of spectral radii attained by graphs is dense in $(\la', \infty)$. As was pointed out in \cite{D89}, Shearer actually proved that the set of spectral radii attained by caterpillar trees\footnote{A caterpillar tree is a tree in which all the vertices are within distance $1$ of its central path.} is already dense in $(\la',\infty)$. Since a caterpillar tree is bipartite (or equivalently, $2$-colorable), we rephrase Shearer's result in terms of smallest eigenvalues.

\begin{theorem}[Shearer \cite{S89}; cf. Theorem~3 of Doob \cite{D89}] \label{thm:caterpillar}
    For every $\la \ge \la'$ and every $\eps > 0$, there exists a caterpillar tree $G$ such that $\la_1(G) \in (-\la - \eps, -\la)$. \qed
\end{theorem}

Jiang and Polyanskii introduced rowing graphs in \cite{JP24} to show that the set of smallest eigenvalues attained by graphs is dense in $(-\infty, -\las)$, filling the gap between $-\la' \approx -2.05817$ and $-\las \approx -2.01980$. We use the Kleene star of $\Sigma$, denoted $\Sigma^*$, to represent the set of all strings of finite length consisting of symbols in $\Sigma$, including the empty string.

\begin{definition}[Rooted graphs and rowing graphs]
    A \emph{rooted graph} is a graph in which one vertex has been distinguished as the root. Given a rooted graph $F$, $n \in \N$, and a string $a \in \N^n$, a \emph{rowing graph} $(F, a)$ is obtained from $F$ by attaching a path $v_0 \dots v_n$ of length $n$ to the root $v_0$ of $F$, and attaching a clique of order $a_i$ to both $v_{i-1}$ and $v_i$ for every $i \in \sset{1,\dots, n}$. As a convention, we regard the rowing graph $(F, a)$ as a rooted graph with root $v_n$.
\end{definition}

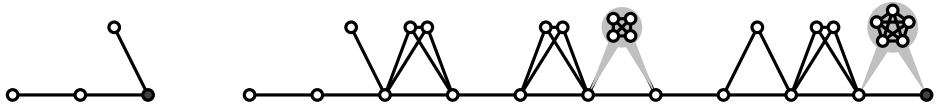
\begin{figure}[t]
    \centering
    \begin{tikzpicture}[very thick, scale=0.45]
        \draw (-11,0) node[vertex]{} -- (-9,0) node[vertex]{} -- (-7,0) node[root-vertex]{} -- (-8,2) node[vertex]{};

        \draw (0,3);
        \draw (-1,2) node[vertex]{} -- (0,0);
        \draw (0.75,2) -- (1.25,2);
        \draw (0,0) -- (0.75,2) node[vertex]{} -- (2,0);
        \draw (0,0) -- (1.25,2) node[vertex]{} -- (2,0);
        
        \draw (4.75,2) -- (5.25,2);
        \draw (4,0) -- (4.75,2) node[vertex]{} -- (6,0);
        \draw (4,0) -- (5.25,2) node[vertex]{} -- (6,0);
        \draw (6,0) -- (7,2) node[vertex]{} -- (8,0);

        \draw[fill=litegray, draw=none] (6,0) -- (7,1.6) -- (8,0) -- (7,2.6) -- cycle;
        \draw[fill=litegray, draw=none] (7, 2) circle (0.6);

        \coordinate (C1) at (7,2);
        \foreach \i/\ang in {1/45, 2/135, 3/225, 4/315}{
            \coordinate (Q\i) at ($(C1)+(\ang:0.36)$);
        }
        \draw (Q1) -- (Q2) -- (Q3) -- (Q4) -- (Q1) -- (Q3);
        \draw (Q2) -- (Q4);
        \foreach \i in {1,2,3,4} {
            \node[vertex] at (Q\i){};
        }

        \draw (10,0) -- (11,2) node[vertex]{} -- (12,0);

        \draw (12.75,2) -- (13.25,2);
        \draw (12,0) -- (12.75,2) node[vertex]{} -- (14,0);
        \draw (12,0) -- (13.25,2) node[vertex]{} -- (14,0);

        \draw[fill=litegray, draw=none] (14,0) -- (15,1.5) -- (16, 0) -- (15,2.7) -- cycle;
        \draw[fill=litegray, draw=none] (15, 2) circle (0.75);

        \coordinate (C2) at (15,2);
        \foreach \i/\ang in {1/90, 2/162, 3/234, 4/306, 5/18}{
            \coordinate (P\i) at ($(C2)+(\ang:0.5)$);
        }
        \draw (P1)--(P2)--(P3)--(P4)--(P5)--cycle;
        \draw (P1)--(P3)--(P5)--(P2)--(P4)--cycle;
        \foreach \i in {1,2,3,4,5} {
            \node[vertex] at (P\i){};
        }

        \draw (-4,0) node[vertex]{} -- (-2,0) node[vertex]{} -- (0,0) node[vertex]{} -- (2,0) node[vertex]{} -- (4,0) node[vertex]{} -- (6,0) node[vertex]{} -- (8,0) node[vertex]{} -- (10,0) node[vertex]{} -- (12,0) node[vertex]{} -- (14,0) node[vertex]{} -- (16,0) node[root-vertex]{};
        \draw (-1,2) node[above]{};
        \draw (-4,0) node[below]{};
        \draw (-2,0) node[below]{};
        \draw (0,0) node[below]{};
        \draw (2,0) node[below]{};
        \draw (4,0) node[below]{};
        \draw (6,0) node[below]{};
        \draw (8,0) node[below]{};
        \draw (10,0) node[below]{};
        \draw (12,0) node[below]{};
        \draw (14,0) node[below]{};
        \draw (16,0) node[below]{};
    \end{tikzpicture}
    \caption{The rooted graph $F$ and a schematic drawing of the rowing graph $(F, 20240125)$.} \label{fig:rowing-graph}
\end{figure}

With this convention, the rowing graph $(F, ab)$ of $F$ can be viewed as the rowing graph $((F, a), b)$ of $(F, a)$, for every $a, b \in \N^*$. See \cref{fig:rowing-graph} for an example of a rooted graph $F$ and a rowing graph of $F$.

To control chromatic numbers, we focus on a specific family of rowing graphs.

\begin{definition}[Binary rowing graphs]
    A \emph{binary rowing graph} of a rooted graph $F$ is a rowing graph $(F, a)$ for which $a \in \sset{0,1}^*$.
\end{definition}

We import the following two ingredients from \cite{JP24}. We point out that although these were proved in \cite{JP24} specifically for the rooted graph $F$ in \cref{fig:rowing-graph}, the proofs over there actually work for arbitrary rooted graphs. Since our arguments differ slightly, we include the adapted proofs in \cref{sec:app}.

\begin{lemma}[Lemma 2.22(c) of Jiang and Polyanskii~\cite{JP24}] \label{lem:l}
    For every $\eps > 0$ there exists $\ell \in \N^+$ such that for every rooted graph $F$ and $a, b \in \N^*$,
    \[
        \la_1(F, a0^\ell) < \la_1(F, a0^\ell b) + \eps.
    \]
\end{lemma}

\begin{lemma}[Lemma 2.22(e) of Jiang and Polyanskii~\cite{JP24}] \label{lem:m}
    For every $\eps > 0$ and every rooted graph $F$, there exists $m \in \N^+$ such that for every $n \ge m$ and every $a \in \N^n$, there exists $k \in \sset{1,\dots, m}$ such that
    \[
        \la_1(F, a_1 \dots a_{k-1} 0 a_k \dots a_n) < \la_1(F, a) + \eps.
    \]
\end{lemma}

We adopt the following handy notation for every rooted graph $F$ and every $a \in \N^*$,
\[
    \la_1(F, a0^\infty) := \lim_{n \to \infty}\la_1(F, a0^n).
\]
Note that the limit exists because $\la_1(F, a0^n)$ decreases as $n$ increases, and the maximum degree of $(F, a0^n)$ is at most that of $(F, a0^2)$.

We need three more ingredients, the first two of which can respectively be seen as variations of \cite[Lemma 2.22(d)]{JP24} and \cite[Lemma 2.22(b)]{JP24} adapted for binary rowing graphs.

\begin{lemma} \label{lem:one-to-three}
    For every rooted graph $F$ and every $a \in \sset{0,1}^*$, if $-\la' < \la_1(F, a01^30^\infty) < -2$, then
    \[
        \la_1(F, a01^3 0^\infty) \le \la_1(F, a10^\infty).
    \]
\end{lemma}

\begin{proposition} \label{lem:cover}
    For the rooted graphs $F_1, F_2, F_3$ in \cref{fig:rooted-graphs}, the intervals
    \[
        \left(\la_1(F_i, 1^8), \la_1(F_i, 0^\infty)\right) \text{ for }i \in \sset{1,2,3}
    \] cover $(-\la', -\las)$.
\end{proposition}

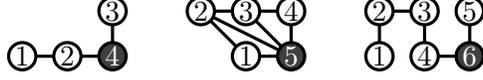
\begin{figure}
    \centering
    \begin{tikzpicture}[very thick, scale=0.3, baseline=(v.base)]
        \coordinate (v) at (0,0);
        \draw (3,0) -- +(0,2) node[vertex]{3};
        \draw (-1,0) node[vertex]{1} -- (1,0) node[vertex]{2} -- (3,0) node[root-vertex]{4};
    \end{tikzpicture}\qquad
    \begin{tikzpicture}[very thick, scale=0.3, baseline=(v.base)]
        \coordinate (v) at (0,0);
        \draw (3,0) -- +(-2,2) -- +(-4,2);
        \draw (3,0) -- (-1,2) node[vertex]{2} -- (1,0) node[vertex]{1} -- (3,0) node[root-vertex]{5} --  (3,2) node[vertex]{4} -- (1,2) node[vertex]{3};
    \end{tikzpicture}\qquad
    \begin{tikzpicture}[very thick, scale=0.3, baseline=(v.base)]
        \coordinate (v) at (0,0);
        \draw (7,0) -- +(0,2) node[vertex]{5};
        \draw (3,0) node[vertex]{1} -- (3,2) node[vertex]{2} -- (5,2) node[vertex]{3} -- (5,0) node[vertex]{4} -- (7,0) node[root-vertex]{6};
    \end{tikzpicture}
    \caption{Rooted graphs $F_1, F_2, F_3$.} \label{fig:rooted-graphs}
\end{figure}

\begin{proposition} \label{lem:neg-2}
    For the rooted graphs $F_1, F_2, F_3$ in \cref{fig:rooted-graphs},
    \[
        \la_1(F_i,0^\infty) < -2 \text{ for }i \in \sset{1,2,3}.
    \]
\end{proposition}

We shall come back to the technical proofs of \cref{lem:one-to-three,lem:cover,lem:neg-2} in the next section.

\begin{proof}[Proof of \cref{thm:main} assuming \cref{lem:one-to-three,lem:cover,lem:neg-2}]
    In view of \cref{thm:caterpillar}, we may assume that $\la \in (\las, \la')$. According to \cref{lem:cover}, there exists a rooted graph $F$ in \cref{fig:rooted-graphs} such that $\la \in (-\la_1(F, 0^\infty), -\la_1(F, 1^8))$. Clearly, every binary rowing graph of $F$ is $3$-colorable.
    
    Fix $\eps > 0$. We assume for the sake of contradiction that no binary rowing graph of $F$ has its smallest eigenvalue in $(-\la - \eps, -\la)$. Let $\ell \in \N^+$ and $m \in \N^+$ be given by \cref{lem:l,lem:m}. Without loss of generality, we may assume that $m + \ell \ge 8$ and $\ell \ge 3$. Define
    \[
        S = \dset{a_1 \dots a_m b_1 \dots b_\ell \in \sset{0,1}^{m+\ell}}{\la_1(F, a_1 \dots a_m b_1 \dots b_\ell c) < -\la \text{ for some }c \in \sset{0,1}^*}.
    \]
    
    Since $\la_1(F, 1^8) < -\la$, we must have $1^{m+\ell} \in S$, and so $S$ is non-empty. Let $ab$ be the minimum element of $S$ under the lexicographical order, where $a \in \sset{0,1}^m$ and $b \in \sset{0,1}^{\ell}$. Let $c \in \sset{0,1}^*$ be the witness of $ab \in S$, that is, $\la_1(F, abc) < -\la$. Since no binary rowing graph of $F$ has its smallest eigenvalue in $(-\la-\eps, -\la)$, we shall repeatedly use the fact that $\la_1(F, abc) + \eps < -\la$.
    
    We claim that $b = 0^\ell$. Indeed, from \cref{lem:m} we obtain $k \in \sset{1,\dots, m}$ such that
    \[
        \la_1(F, a_1 \dots a_{k-1} 0a_k \dots a_m b c) < \la_1(F, a b c) + \eps < -\la,
    \]
    and so $a_1\dots a_{k-1}0a_k \dots a_mb_1\dots b_{\ell-1} \in S$. By the minimality of $ab$ in $S$, we must have $a_k = \dots = a_m = b_1 = \dots = b_\ell = 0$.

    In view of the claim, we get from \cref{lem:l} that
    \begin{equation} \label{eqn:fab}
        \la_1(F, ab) \le \la_1(F, abc) + \eps < -\la.
    \end{equation}
    Since $\la_1(F, 0^{m+\ell}) > \la_1(F, 0^\infty) > -\la$, we must have $a \neq 0^m$. Let $k \in \sset{1,\dots,m}$ be the last index of $a$ for which $a_k = 1$. Set $\tilde{a} = a_1 \dots a_{k-1}$, and set $n := m+\ell-k$. Recall that $\ell \ge 3$, and so $n \ge 3$. By the minimality of $ab$, which is equal to $\tilde{a}10^n$, in $S$, we know that $\tilde{a}01^30^{n-3} \in \sset{0,1}^{m+\ell} \setminus S$, and, in particular, $\la_1(F, \tilde{a}01^30^\infty) \ge -\la$. Note that
    \[
        -\la' < -\la \le \la_1(F, \tilde{a}01^30^\infty) < \la_1(F, 0^\infty) < -2,
    \]
    where the last inequality comes from \cref{lem:cover-concrete}. We apply \cref{lem:one-to-three} to obtain that
    \[
        \la_1(F, \tilde{a}01^30^\infty) \le \la_1(F, \tilde{a}10^\infty) \le \la_1(F, ab),
    \]
    which contradicts with \eqref{eqn:fab}.
\end{proof}

\section{Proofs of \cref{lem:one-to-three,lem:cover,lem:neg-2}} \label{sec:technical}

The following linear algebraic lemma characterizes $\la_1(F, a0^\infty)$. Denote by $E_{v,v}$ the unit matrix where the $(v,v)$-entry with value $1$ is the only nonzero entry. We write $A \succeq 0$ when $A$ is positive-semidefinite and $A \succ 0$ when $A$ is positive-definite.


\begin{lemma}[Lemma 26 of Acharya and Jiang \cite{AJ25}] \label{lem:path-ext}
    For every rooted graph $G$, and every $x \ge 2$, \[
        \la_1(G, 0^\infty) \ge -x \text{ if and only if }
        A_G + xI - yE_{v,v} \succeq 0,
    \]
    where $v$ is the root of $G$, and $y = x / 2 - \sqrt{x^2/4-1}$. \qed
\end{lemma}

Notice that each diagonal entry of $A_G + xI - yE_{v,v}$ increases as $x$ increases. We have the following version of \cref{lem:path-ext} with strict inequalities.

\begin{corollary} \label{lem:path-ext-strict}
    For every rooted graph $G$, and every $x > 2$, \[
        \la_1(G, 0^\infty) > -x \text{ if and only if }
        A_G + xI - yE_{v,v} \succ 0,
    \]
    where $v$ is the root of $G$, and $y = (x / 2 - \sqrt{x^2/4-1})$. \qed
\end{corollary}

We are in the position to characterize $\la_1(F, a01^30^\infty)$ and $\la_1(F, a10^\infty)$ in \cref{lem:one-to-three}.

\begin{lemma} \label{lem:path-ext-plus}
    For every rooted graph $G$, and every $x > 2$,
    \begin{gather*}
        \la_1(G, 01^30^\infty) \ge -x \text{ if and only if }A_G + xI - \alpha E_{v,v}\succeq 0, \\
        \la_1(G, 10^\infty) \ge -x \text{ if and only if }A_G + xI - \beta E_{v,v}\succeq 0,
    \end{gather*}
    where $v$ is the root of $G$,
    \begin{equation}\begin{gathered} \label{eqn:aby}
        \alpha = \frac{(x^5 - 5x^3 + 2x^2 + 3x) y - (x^6 - 7x^4 + 4x^3 + 9x^2 - 6x - 1)}{(x^6 - 7x^4 + 4x^3 + 9x^2 - 6x - 1) y - (x^7 - 9x^5 + 6x^4 + 19x^3 - 20x^2- 3x + 6)}, \\
        \beta = \frac{y - (2x-2)}{x y - (x^2-1)},
        \text{ and }y = x / 2 - \sqrt{x^2/4-1}.
    \end{gathered}\end{equation}
\end{lemma}

\begin{proof}
    Suppose that $x > 2$. By \cref{lem:path-ext}, we know that
    \[
        \la_1(G, 01^30^\infty) \ge -x \text{ if and only if }
        M_1 := A_{(G, 01^3)} + xI - yE_{u,u} \succeq 0,
    \]
    where $u$ is the root of $(G, 01^3)$. We partition the matrix $M_1$ into the following blocks
    \[
        M_1 = \begin{pmatrix}
            A_{G} + xI & B_1 \\
            B_1^\intercal & C_1
        \end{pmatrix},
        \text{ where }
        C_1 = A_{(\bullet, 1^3)} + xI - yE_{u,u}.
    \]
    Since $(\bullet, 1^30^n)$ is a line graph for every $n \in \N$, we have $\la_1(\bullet, 1^30^\infty) \ge -2 > -x$, and so $C_1 \succ 0$ via \cref{lem:path-ext-strict}. Therefore $M_1 \succeq 0$ if and only if the Schur complement $A_G + xI - B_1C_1^{-1}B_1^\intercal \succeq 0$. Let $v_0$ be the vertex in $V(G, 0) \setminus V(G)$. Since the only nonzero entry of $B_1$ is its $(v, v_0)$ entry, the matrix $B_1C_1^{-1}B_1^\intercal$ simplifies to $(C_1^{-1})_{v_0, v_0}E_{v,v}$. Cramer's rule yields $(C_1^{-1})_{v_0,v_0} = \det C_1' / \det C_1$, where $C_1'$ is obtained from $C_1$ by removing the $v_0$-th row and column. A routine calculation yields
    \begin{align*}
        \det C_1' & = -(x^5 - 5x^3 + 2x^2 + 3x) y + (x^6 - 7x^4 + 4x^3 + 9x^2 - 6x - 1), \\
        \det C_1 & = -(x^6 - 7x^4 + 4x^3 + 9x^2 - 6x - 1) y + (x^7 - 9x^5 + 6x^4 + 19x^3 - 20x^2- 3x + 6),
    \end{align*}
    and so $(C_1^{-1})_{v_0,v_0} = \alpha$.

    By \cref{lem:path-ext}, we know that $\la_1(G, 10^\infty) \ge -x$ if and only if $M_2 := A_{(G, 1)} + xI - yE_{v_1,v_1} \succeq 0$, where $v_1$ is the root of $(G, 1)$. We partition $M_2$ into blocks
    \[
        M_2 = \begin{pmatrix}
            A_G + xI & B_2 \\
            B_2^\intercal & C_2
        \end{pmatrix},
        \text{ where }C_2 = A_{(\bullet, 0)} + xI - yE_{v_1,v_1}.
    \]
    A similar argument shows that $M_2 \succeq 0$ if and only if $A_G + xI - B_2C_2^{-1}B_2^\intercal \succeq 0$. Let $v_2$ be the other vertex in $V(G,1) \setminus V(G)$. Since the only nonzero entries of $B_2$ are its $(v,v_1)$ and $(v,v_2)$ entries, the matrix $B_2C_2^{-1}B_2^\intercal$ simplifies to $\left(\sum_{u_1,u_2} (C_2^{-1})_{u_1,u_2}\right) E_{v,v}$. A routine calculation yields $\sum_{u_1,u_2} (C_2^{-1})_{u_1,u_2} = \beta$.
\end{proof}

To prove \cref{lem:one-to-three}, we simply compare $\alpha$ and $\beta$ defined in \cref{lem:path-ext-plus}.

\begin{proof}[Proof of \cref{lem:one-to-three}]
    Let $G$ be the rooted graph $(F, a)$, and let $v$ be the root of $G$. For every $x > 2$, we know through \cref{lem:path-ext-plus} that $\la_1(G, 01^30^\infty) \ge -x$ if and only if $A_G + x I - \alpha E_{v,v} \succeq 0$, and $\la_1(G, 10^\infty) \ge -x$ if and only if $A_G + x I - \beta E_{v,v} \succeq 0$, where $\alpha$, $\beta$, and $y$ are defined as in \eqref{eqn:aby}.
    
    Since $E_{v,v} \succeq 0$, it suffices to show that $\alpha \ge \beta$ for $x \in (2, \la')$. Note that $x = y + 1/y$, and $x \in (2,\la')$ if and only if $y \in (\varphi^{-1/2},1)$, where $\varphi = (\sqrt 5 + 1)/2$. Replacing $x$ by $y + 1/y$ in $\alpha - \beta$, we obtain the rational function
    \[
        \frac{y(y-1)^2(y^{12} - 2y^{11} + 2y^{10} + 2y^9 - 5y^8 + 6y^7 - 4y^6 + 2y^5 + 3y^4 - 6y^3 + 4y^2 - 1)}{(y^2-y+1)(-y^{10} + y^9 + y^8 - 2y^7 + y^6 - y^5 +y^4 + 2y^3 - 3y^2 +y+ 1)}.
    \]
    It is routine to check via Sturm's theorem that each polynomial factor in the rational function is positive on $(7/9, 1) \approx (0.77778, 1)$, which contains $(\varphi^{-1/2},1) \approx (0.78615, 1)$.
\end{proof}

To prove \cref{lem:cover}, we gather the following computations about the smallest eigenvalues of binary rowing graphs of rooted graphs in \cref{fig:rooted-graphs}.

\begin{proposition}[Hoffman~\cite{H72}] \label{lem:f1}
    For every $n \in \N$, as $n \to \infty$, the smallest eigenvalue of the binary rowing graph $(F_1, 0^n)$ of the rooted graph $F_1$ in \cref{fig:rooted-graphs} decreases to $-\las$. \qed
\end{proposition}

\begin{proposition} \label{lem:cover-concrete}
    The rowing graphs of the rooted graphs $F_1, F_2, F_3$ in \cref{fig:rooted-graphs} satisfy the following properties:
    \begin{align*}
        \la_1(F_1, 1^8) & < -\tfrac{136}{67}, & \la_1(F_2, 1^7) & < -\tfrac{178}{87}, & \la_1(F_3, 1^4) & < -\la', \\
        \la_1(F_1, 0^\infty) & = -\las, & \la_1(F_2, 0^\infty) & \in (-\tfrac{136}{67}, -2), & \la_1(F_3, 0^\infty) & \in (-\tfrac{178}{87}, -2).
    \end{align*}
\end{proposition}

\begin{figure}
    \centering
    \begin{tikzpicture}[very thick, scale=0.45, baseline=(v.base)]
        \coordinate (v) at (0,0);
        \draw (19,0) -- (3,0) -- +(0,2) node[vertex, label={$-52$}]{};
        \draw (-1,0) node[vertex, label={below:$-34$}]{} -- (1,0) node[vertex, label={below:$-69$}]{} -- (3,0) node[vertex, label={below:$106$}]{} -- (5,2)node[vertex, label={$-12$}]{} -- (5,0)node[vertex, label={below:$-82$}]{} -- (7,2)node[vertex, label={$9$}]{} -- (7,0)node[vertex, label={below:$63$}]{} -- (9,2)node[vertex, label={$-7$}]{} -- (9,0)node[vertex, label={below:$-48$}]{} -- (11,2)node[vertex, label={$6$}]{} -- (11,0)node[vertex, label={below:$36$}]{} -- (13,2)node[vertex, label={$-5$}]{} -- (13,0)node[vertex, label={below:$-26$}]{} -- (15,2)node[vertex, label={$4$}]{} -- (15,0)node[vertex, label={below:$18$}]{} -- (17,2)node[vertex, label={$-4$}]{} -- (17,0)node[vertex, label={below:$-10$}]{} -- (19,2)node[vertex, label={$3$}]{} -- (19,0)node[root-vertex, label={below:$3$}]{};
    \end{tikzpicture}\par\medskip
    \begin{tikzpicture}[very thick, scale=0.45, baseline=(v.base)]
        \coordinate (v) at (0,0);
        \draw (3,0) -- (17,0);
        \draw (3,0) -- (5,2)node[vertex, label={$9$}]{} -- (5,0)node[vertex, label={below:$52$}]{} -- (7,2)node[vertex, label={$-7$}]{} -- (7,0)node[vertex, label={below:$-38$}]{} -- (9,2)node[vertex, label={$5$}]{} -- (9,0)node[vertex, label={below:$27$}]{} -- (11,2)node[vertex, label={$-4$}]{} -- (11,0)node[vertex, label={below:$-19$}]{} -- (13,2)node[vertex, label={$3$}]{} -- (13,0)node[vertex, label={below:$13$}]{} -- (15,2)node[vertex, label={$-3$}]{} -- (15,0)node[vertex, label={below:$-7$}]{} -- (17,2)node[vertex, label={$2$}]{} -- (17,0)node[root-vertex, label={below:$2$}]{};
        \draw (3,0) -- +(-2,2) -- +(-4,2);
        \draw (3,0) -- (-1,2) node[vertex, label={$14$}]{} -- (1,0) node[vertex, label={below:$28$}]{} -- (3,0) node[vertex, label={below:$-71$}]{} --  (3,2) node[vertex, label={$28$}]{} -- (1,2) node[vertex, label={$14$}]{};
    \end{tikzpicture}\qquad
    \begin{tikzpicture}[very thick, scale=0.45, baseline=(v.base)]
        \coordinate (v) at (0,0);
        \draw (7,0) -- (15,0);
        \draw (7,0) -- (9,2)node[vertex, label={$-5$}]{} -- (9,0)node[vertex, label={below:$-24$}]{} -- (11,2)node[vertex, label={$4$}]{} -- (11,0)node[vertex, label={below:$15$}]{} -- (13,2)node[vertex, label={$-3$}]{} -- (13,0)node[vertex, label={below:$-9$}]{} -- (15,2)node[vertex, label={$3$}]{} -- (15,0)node[root-vertex, label={below:$3$}]{};
        \draw (7,0) -- +(0,2) node[vertex, label={$-17$}]{};
        \draw (3,0) node[vertex, label={below:$6$}]{} -- (3,2) node[vertex, label={$-12$}]{} -- (5,2) node[vertex, label={$18$}]{} -- (5,0) node[vertex, label={below:$-26$}]{} -- (7,0) node[vertex, label={below:$35$}]{};
    \end{tikzpicture}
    \caption{Test vectors $z_1, z_2, z_3$ on $(F_1, 1^8), (F_2, 1^7), (F_3, 1^4)$} \label{fig:test-vectors}
\end{figure}
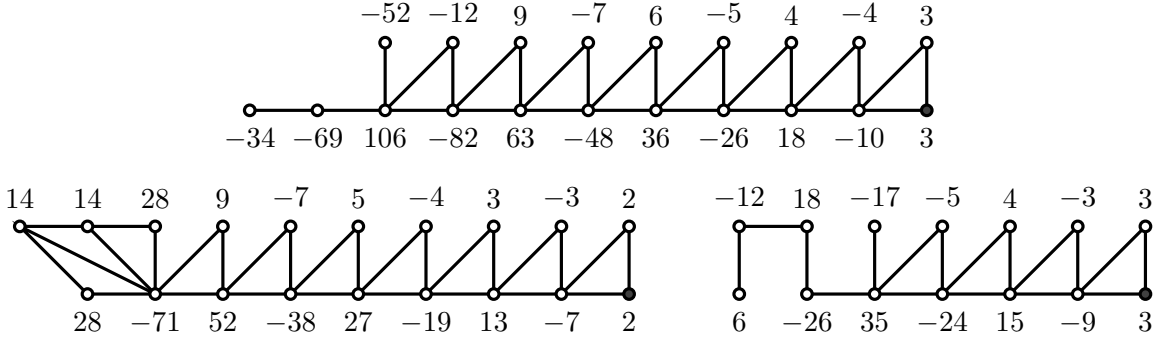

\begin{proof}
    We assign test vectors $z_1, z_2, z_3$ on $(F_1, 1^8),(F_2, 1^7),(F_3, 1^4)$ as in \cref{fig:test-vectors}. The corresponding Rayleigh quotients are $-\tfrac{72334}{35635}$, $-\tfrac{4315}{2109}$, $-\tfrac{1875}{911}$, which are less than $-\tfrac{136}{67}$, $-\tfrac{178}{87}$, $-\la'$ respectively.
    
    According to \cref{lem:f1}, we know that $\la_1(F_1, 0^\infty) = -\las$. Set $x_2 = \tfrac{136}{67}$ and $x_3 = \tfrac{178}{87}$. According to \cref{lem:path-ext-strict}, to prove that $\la_1(F_2, 0^\infty) > -x_2$ and $\la_1(F_3, 0^\infty) > -x_3$, it suffices to show that
    \[
        A_{F_2} + x_2I - y_2E_{v_2,v_2} \succ 0 \text{ and }
        A_{F_3} + x_3I - y_3E_{v_3,v_3} \succ 0,
    \]
    where $y_2 = x_2/2 - \sqrt{x_2^2/4-1}$, $y_3 = x_3/2 - \sqrt{x_3^2/4-1}$, and $v_2$ and $v_3$ are the roots of $F_2$ and $F_3$ respectively. Since $y_2 \approx 0.84151$, $y_3 \approx 0.80734$, $E_{v_2,v_2} \succeq 0$ and $E_{v_3,v_3} \succeq 0$, it suffices to show that
    \[
        M_2 := A_{F_2} + x_2I - y_2'E_{v_2,v_2} \succ 0 \text{ and }
        M_3 := A_{F_3} + x_3I - y_3'E_{v_3,v_3} \succ 0,
    \]
    where $y_2' = \tfrac{85}{101} \approx 0.84158$ and $y_3' = \tfrac{285}{353} \approx 0.80737$. Note that
    \[
        M_2 = \begin{pmatrix}
            \tfrac{136}{67} & 1 & & & 1 \\
            1 & \tfrac{136}{67} & 1 & & 1 \\
             & 1 & \tfrac{136}{67} & 1 & 1 \\
             & & 1 & \tfrac{136}{67} & 1 \\
             1 & 1 & 1 & 1 & \tfrac{8041}{6767}
        \end{pmatrix} \text{ and }
        M_3 = \begin{pmatrix}
            \tfrac{178}{87} & 1 &&&& \\
            1 & \tfrac{178}{87} & 1 &&&\\
            & 1 & \tfrac{178}{87} & 1 &&\\
            && 1 & \tfrac{178}{87} && 1\\
            &&&& \tfrac{178}{87} & 1\\
            &&& 1 & 1 & \frac{38039}{30711}
        \end{pmatrix}.
    \]
    
    Notice that the leading principal submatrix of $M_2$ of order $4$ and that of $M_3$ of order $5$ are positive definite, because both submatrices are of the form $A_G + xI$, where $G$ is a graph of maximum degree $2$, and $x > 2$. It is routine to compute that $\det M_2 = \tfrac{50854155}{136362635807}
    > 0$ and $\det M_3 = \tfrac{23578825817}{153070048956177} > 0$.

    According to \cref{lem:path-ext}, to prove that $\la_1(F_2, 0^\infty) < -2$ and $\la_1(F_3, 0^\infty) < -2$, it suffices to show that
    \[
        M_2' := A_{F_2} + 2I - E_{v_2,v_2} \not\succeq 0 \text{ and }M_3' := A_{F_3} + 2I - E_{v_3,v_3} \not\succeq 0.
    \]
    Note that \[
        M_2' = \begin{pmatrix}
            2 & 1 &&& 1 \\
            1 & 2 & 1 && 1 \\
            & 1 & 2 & 1 & 1 \\
            && 1 & 2 & 1 \\
            1 & 1 & 1 & 1 & 1
        \end{pmatrix} \text{ and }
        M_3' = \begin{pmatrix}
            2 & 1 &&&& \\
            1 & 2 & 1 &&& \\
            & 1 & 2 & 1 && \\
            && 1 & 2 && 1 \\
            &&&& 2 & 1 \\
            &&& 1 & 1 & 1
        \end{pmatrix}
    \]
    Finally, it is routine to compute that $\det M_2' = -1$ and $\det M_3' = -3$.
\end{proof}

\begin{proof}[Proof of \cref{lem:cover,lem:neg-2}]
    They follow from \cref{lem:cover-concrete} immediately.
\end{proof}

\section{Spherical two-distance sets} \label{sec:two-distance}

A \emph{spherical two-distance set} in $\R^d$ is a collection of unit vectors whose pairwise inner products lie in $\sset{\alpha,\beta}$. There was some partial success specifically in the regime where $-1 \le \beta < 0 \le \alpha < 1$. For fixed $\alpha$ and $\beta$, denote by $\nabd$ the maximum size of such a set in $\R^d$.

When $\beta = -\alpha$, the problem specializes to the famous question of \emph{equiangular lines} with fixed angles, whose high-dimensional behavior is now understood thanks to the breakthrough of Jiang, Tidor, Yao, Zhang and Zhao \cite{JTYZZ21}. Outside the equiangular setting, the landscape is much richer and remain largely open.

Recently, the limit $\lim_{d\to\infty}\nabd/d$ was determined in \cite{JP24,JTYZZ23} when $p \le 2$ or $\la < \las$, where the parameters are defined by
\[
    p = \left\lfloor{-\frac{\alpha}{\beta}}\right\rfloor + 1 \quad\text{and}\quad \la = \frac{1-\alpha}{\alpha-\beta}.
\]
In \cite{JTYZZ23}, a spherical two-distance set with prescribed inner products $\alpha$ and $\beta$ in $\R^d$ was constructed, and it was conjectured to provide the optimal $\nabd$ up to an additive error of $o(d)$. The limit $\lim_{d\to\infty}\nabd / d$ is believed to depend on the following spectral graph theoretic quantity:
\[
    k_p(\la) = \inf\dset{\frac{\abs{G}}{\mult(-\la, G)}}{G \in \Gp(\la)},
\]
where $\mult(-\la, G)$ denotes the multiplicity of $-\la$ as an eigenvalue of $G$.

\begin{conjecture}[Conjecture 1.11 of Jiang et al.\ \cite{JTYZZ23}] \label{conj:main}
    Fix $-1 \le \beta < 0 \le \alpha < 1$. Set $\la = (1-\alpha)/(\alpha-\beta)$ and $p = \lfloor -\alpha/\beta \rfloor + 1$. Then
    \[
        \nabd = \begin{cases}
            \dfrac{k_p(\la)d}{k_p(\la)-1} + o(d) & \text{if }k_p(\la) < \infty, \\
            d + o(d) & \text{otherwise}.
        \end{cases}
    \]
\end{conjecture}

To prove \cref{conj:main}, a forbidden-subgraph approach was discovered in \cite{JTYZZ23} that reduces the estimation of $\nabd$ to that of the following quantity. See \cite[Section 5]{JTYZZ23} for the reduction.

\begin{definition}[Definition 5.2 of Jiang et al.\ \cite{JTYZZ23}] \label{def:mpf}
    Given $p \in \N^+$ and a family $\F$ of signed graphs, let $M_{p, \F}(\la, N)$ be the maximum possible value of $\mult(-\la, G)$ over all $p$-colorable signed graphs $G$ on at most $N$ vertices that do not contain any member of $\F$ as an induced subgraph and satisfy $\la_{p+1}(G) \ge -\la$. Here $\la_{p+1}(G)$ denotes the $(p+1)$-th smallest eigenvalue of $G$.
\end{definition}

\begin{conjecture}[Conjecture 5.4 of Jiang et al.\ \cite{JTYZZ23}] \label{conj:forb}
    For every $\la > 0$ and $p \in \N^+$, there exists a finite family $\F$ of signed graphs with $\la_1(F) < -\la$ for each $F \in \F$ such that
    \[
        M_{p,\F}(\la, N) \le \begin{cases}
            N / k_p(\la) + o(N) & \text{if }k_p(\la) < \infty, \\
            o(N) & \text{otherwise}.
        \end{cases}
    \]
\end{conjecture}

When $\la < \las$, according to \cref{thm:signed-forb}, the family $\F$ in \cite{JP24} was taken as a finite forbidden subgraph characterization of the family of signed graphs with smallest eigenvalue at least $-\la$ (within the class of signed graphs). In view of \cref{def:mpf}, for this specific $\F$, the quantity simplifies to
\[
    M_{p, \F}(\la, N) = \max\dset{\mult(-\la, G)}{\abs{G} \le N \text{ and }G \in \Gp(\la)},
\]
which in turn confirms \cref{conj:forb} through a simple argument. We refer the readers to \cite[Section 4]{JP24} for more details.

Since a $p$-colorable graph $G$ never contains a subgraph $F$ that is not $p$-colorable, it is more economical to choose $\F$ to be a forbidden subgraph characterization of $\Gp(\la)$ within $\Gp$. This is the motivation behind \cref{prob:main}. Our \cref{thm:main} therefore shows that restricting to $p$-colorable signed graphs yields no new finite forbidden subgraph characterizations once $p \ge 3$.

\section*{Acknowledgements} We thank Xiaonan Liu for valuable early-stage discussions and the referee for a careful and insightful review. We also gratefully acknowledge travel support from the LSU Provost's Fund for Innovation in Research --- Seminar/Collaborator Support Fund.

\bibliographystyle{plain}
\bibliography{3colorable}

\appendix

\section{Proofs of \cref{lem:l,lem:m}} \label{sec:app}

In the proofs of \cref{lem:l,lem:m}, we work with vectors $x$ on the vertex set of a rowing graph, whose path attached to its rooted graph is denoted by $v_0 v_1 \dots$, and we abuse notation and write $x_i$ in place of $x_{v_i}$.

\begin{proof}[Proof of \cref{lem:l}]
    Take $\ell \in \N^+$ such that $\ell > 3 / \eps$ and $\la_1(\bullet, 0^\ell) = 2\cos((\ell-1)\pi/\ell) < -2 + \eps / 3$. Let $v_0 \dots v_{m + \ell + n}$ denote the path attached to the root of $F$ in the rowing graph $(F, a0^\ell b)$, where $a \in \N^m$ and $b \in \N^n$, and let $x \colon V(F, a0^\ell b) \to \R$ be a unit eigenvector associated with the smallest eigenvalue of $(F, a0^\ell b)$. Choose $k \in \sset{0,\dots,\ell-1}$ such that $x_{m + k}x_{m + k + 1}$ reaches the minimum in absolute value. In particular, using the inequality $\abs{x_{m+i}x_{m+i+1}} \le (x_{m+i}^2 + x_{m+i+1}^2)/2$, we obtain
    \begin{equation} \label{eqn:rowing-b}
        \abs{x_{m + k}x_{m + k + 1}} \le \frac{1}{\ell}\sum_{i=0}^{\ell-1}\abs{x_{m+i}x_{m+i+1}} \le \frac{1}{\ell}\sum_{i=0}^\ell x_{m+i}^2 \le \frac{1}{\ell} < \frac{\eps}{3}.
    \end{equation}

    Notice that removing the edge $v_{m+k}v_{m+k+1}$ disconnects $(F, a0^\ell b)$ into two subgraphs, one of which is $(F, a0^k)$, while the other is $(\bullet, 0^{\ell-k-1}b)$,  which is the line graph of a caterpillar tree. Clearly
    \begin{equation} \label{eqn:rowing-b-1}
        \la_1(F, a0^k) \ge \la_1(F, a0^\ell).
    \end{equation}
    Together with $\la_1(F, a0^\ell) \le \la_1(\bullet, 0^\ell) < -2 + \eps/3$ and $\la_1(\bullet, 0^{\ell-k-1}b) \ge -2$, we obtain
    \begin{equation} \label{eqn:rowing-b-2}
        \la_1(\bullet, 0^{\ell-k-1}b) > \la_1(F, a0^\ell) - \eps/3.
    \end{equation}

    Let $x_L$ and $x_R$ be the unit eigenvector $x$ restricted to $V(F, a0^k)$ and $V(\bullet, 0^{\ell-k-1}b)$ respectively. Finally, we bound the smallest eigenvalue of $(F, a0^\ell b)$ as follows:
    \begin{align*}
        \la_1(F, a0^\ell b) & \stackrel{\phantom{\rowingrefs}}{=} x^\intercal A_{(F, a0^\ell b)}x \\
        & \stackrel{\phantom{\rowingrefs}}{=} x_L^\intercal A_{(F, a0^k)} x_L + 2x_{m+k}x_{m+k+1} + x_R^\intercal A_{(\bullet, 0^{\ell-k-1}b)} x_R \\
        & \stackrel{\phantom{\rowingrefs}}{\ge} \la_1(F, a0^k) x_L^\intercal x_L + 2x_{m+k}x_{m+k+1} + \la_1(\bullet, 0^{\ell-k-1}b) x_R^\intercal x_R \\
        & \stackrel{\rowingrefs}{>} \big(\la_1(F, a0^\ell) - \eps/3\big)\big(x_L^\intercal x_L + x_R^\intercal x_R\big) - 2\eps/3 \\
        & \stackrel{\phantom{\rowingrefs}}{=} \la_1(F, a0^\ell) - \eps. \qedhere
    \end{align*}
\end{proof}

\begin{proof}[Proof of \cref{lem:m}]
    By Weyl's inequality, we obtain that for every $a \in \N^*$,
    \begin{equation} \label{eqn:rowing-d-1}
        \la_1(F, a) \ge \la_1(F) + \la_1(\bullet, a) \ge \la_1(F) - 2.
    \end{equation}
    Take $m \in \N^+$ such that $m > (4 - \la_1(F))/\eps$. Suppose that $n \ge m$. Let $v_0 \dots v_n$ denote the path attached to the root of $F$ in the rowing graph $(F, a)$, where $a \in \N^n$, and let $x \colon V(F, a) \to \R$ be a unit eigenvector associated with the smallest eigenvalue of $(F, a)$. Choose $k \in \sset{1, \dots, m}$ such that $x_{k-1}$ reaches the minimum in absolute value. In particular,
    \begin{equation} \label{eqn:rowing-d-2}
        x_{k-1}^2 \le \frac{1}{m} \sum_{i=0}^{m-1}x_i^2 \le \frac{1}{m} < \frac{\eps}{4 - \la_1(F)}.
    \end{equation}

    Let $v_0 \dots v_{k-1}v_* v_{k} \dots v_n$ denote the path attached to the root of $F$ in the rowing graph $(F, \tilde{a})$, where $\tilde{a} = a_1 \dots a_{k-1} 0 a_{k} \dots a_n$. We naturally view the vertex set of $(F, \tilde{a})$ as $V(F, a) \cup \{ v_*\}$, and we extend the unit eigenvector $x\colon V(F, a) \to \R$ to $\tilde{x}\colon V(F, \tilde{a}) \to \R$ by setting $\tilde{x}_* = x_{k-1}$. The Rayleigh principle says that $\la_1(F, \tilde{a})$ is at most
    \begin{multline*}
        \frac{\tilde{x}^\intercal A_{F, \tilde{a}} \tilde{x}}{\tilde{x}^\intercal \tilde{x}}
        = \frac{x^\intercal A_{F, a} x + 2x_{k-1}^2}{x^\intercal x + x_{k-1}^2}
        = \frac{\la_1(F, a) + 2x_{k-1}^2}{1 + x_{k-1}^2} \\
        \le \la_1(F, a) + (2-\la_1(F, a))x_{k-1}^2
        \stackrel{\rowingdrefs}{<} \la_1(F, a) + \eps. \qedhere\kern-10pt
    \end{multline*}
\end{proof}

\end{document}